\documentclass[a4paper]{article}

\usepackage{amsmath}
\usepackage{amssymb}
\usepackage{amsthm}
\usepackage{enumitem}
\usepackage{geometry}
\usepackage{graphicx}
\usepackage{stmaryrd}
\usepackage{tikz-cd}

\theoremstyle{definition}
\newtheorem{dfn}{Definition}[section]

\theoremstyle{plain}
\newtheorem{prop}[dfn]{Proposition}
\newtheorem{thm}[dfn]{Theorem}
\newtheorem*{thm*}{Theorem}
\newtheorem{cor}[dfn]{Corollary}

\numberwithin{equation}{section}

\title{REDUCTION THEOREM FOR SUPPORT $\tau$-TILTING MODULES OVER GROUP ALGEBRAS}
\author{NAOYA HIRAMAE\thanks{Department of Mathematics, Kyoto University, Kitashirakawa Oiwake-cho, Sakyo-ku, Kyoto 606-8502, Japan. E-mail: hiramae.naoya.58r@st.kyoto-u.ac.jp}}
\date{}

\begin{document}

\maketitle

\begin{abstract}
    \indent In studying the structure of derived categories of module categories of group algebras or their blocks, it is fundamental to classify support $\tau$-tilting modules. Koshio and Kozakai (\cite{KK1}, \cite{KK2} and \cite{KK3}) showed that the structure of support $\tau$-tilting modules over blocks of finite groups can be reduced to that of their subgroups under suitable conditions. We show that the `quotient reduction' is also valid under suitable conditions.\\
\end{abstract}

\section{Introduction}

Modular representation theory of finite groups was initiated by Brauer around 1920, yet many fundamental problems are still open. Among them ``Broué's Abelian Defect Group Conjecture'', a conjecture concerning derived equivalences of modular representations of finite groups, is well-known \cite{B}. Derived equivalent algebras of finite dimension are known to be classified by tilting complexes \cite{R}. Among all tilting complexes, two-term tilting complexes are in bijection with support $\tau$-tilting modules \cite{AIR}. Hence, by examining the support $\tau$-tilting modules over group algebras we can partially classify their derived equivalence classes. Moreover, support $\tau$-tilting modules correspond bijectively with categorical objects such as functorially finite torsion classes \cite{AIR}, two-term simple-minded collections \cite{KY}, intermediate $t$-structures \cite{BY}, left finite semibricks \cite{A} and more. Therefore, it is important to classify support $\tau$-tilting modules over group algebras in the study of modular representations not only because it is expected to lead to the resolution of Broué's Abelian Defect Group Conjecture, but also because it implies the classification of many other categorical objects over group algebras.\\
\indent In \cite{KK1}, \cite{KK2} and \cite{KK3}, Koshio and Kozakai found that under suitable conditions the structure of support $\tau$-tilting modules on blocks of finite groups coincides with that on covered blocks of their subgroups. Moreover, they described explicitly the correspondence of left finite semibricks induced by the correspondence of support $\tau$-tilting modules. In this paper, we first refine the reduction theorem for support $\tau$-tilting modules in \cite{EJR} and establish the reduction theorem for semibricks. Then applying it to the case of group algebras, we show that the structure of support $\tau$-tilting modules on blocks of finite groups coincides with that on blocks of their quotient groups under suitable conditions, and describe the induced correspondence of semibricks. More precisely, we have the following.\\

\begin{thm}[See Theorem \ref{mainthm}]\label{intro1}
    Let $k$ be an algebraically closed field of positive characteristic $p$ and $G$ a finite group. For each $p$-subgroup $N$ of the center $Z(G)$ of $G$, the following assertions hold.
    \begin{enumerate}
        \setlength{\parskip}{0cm}
        \setlength{\itemsep}{0cm}
        \item There exists a bijection between the set of support $\tau$-tilting $kG$-modules and the set of support $\tau$-tilting $k(G/N)$-modules.
        \item The natural embedding $k(G/N)\textrm{-}\mathrm{mod}\rightarrow kG\textrm{-}\mathrm{mod}$ induces a bijection between the set of left finite semibricks over $k(G/N)$ and the set of left finite semibricks over $kG$.
        \item The bijections of (a) and (b) are compatible with the Asai's bijection between support $\tau$-tilting modules and left finite semibricks.
        \item The natural surjection $kG\rightarrow k(G/N)$ induces a bijection between the set of blocks of $kG$ and the set of blocks of $k(G/N)$. Moreover, the same assertions as in (a), (b) and (c) for the correspondent blocks under this bijection also hold true. \\
    \end{enumerate}
\end{thm}

\indent As far as the principal blocks are concerned, we can relax the assumption of Theorem \ref{intro1} and find that the central extensions of finite groups do not affect the structure of support $\tau$-tilting modules or left finite semibricks over the principal blocks.\\
 
\begin{thm}[See Theorem \ref{genmainthm}]\label{intro2}
    Let $k$ be an algebraically closed field of positive characteristic and $G$ a finite group. Then the following assertions hold.
    \begin{enumerate}
        \setlength{\parskip}{0cm}
        \setlength{\itemsep}{0cm}
        \item There exists a bijection between the set of support $\tau$-tilting $B_0(kG)$-modules and the set of support $\tau$-tilting $B_0(k(G/Z(G)))$-modules.
        \item The natural embedding $k(G/Z(G))\textrm{-}\mathrm{mod}\rightarrow kG\textrm{-}\mathrm{mod}$ induces a bijection between the set of left finite semibricks over $B_0(k(G/Z(G)))$ and the set of left finite semibricks over $B_0(kG)$.
        \item The bijections of (a) and (b) are compatible with the Asai's bijection between support $\tau$-tilting modules and left finite semibricks.\\
    \end{enumerate}
\end{thm}

\noindent\textbf{Notation.} Throughout this paper, $k$ always denotes an algebraically closed field. For a finite dimensional $k$-algebra $\Lambda$, we denote the opposite algebra of $\Lambda$ by $\Lambda^{\mathrm{op}}$, the category of finite dimensional left $\Lambda$-modules by $\Lambda\textrm{-mod}$, the full subcategory of $\Lambda\textrm{-mod}$ consisting of all finite dimensional left projective $\Lambda$-modules by $\Lambda\textrm{-proj}$, the homotopy category of bounded complexes of finite dimensional left projective $\Lambda$-modules by $K^b(\Lambda\textrm{-}\mathrm{proj})$, the derived category of bounded complexes of finite dimensional left $\Lambda$-modules by $D^b(\Lambda\textrm{-mod})$. For $M\in\Lambda\textrm{-mod}$, we denote the number of nonisomorphic indecomposable direct summands of $M$ by $|M|$, the $k$-dual of $M$ by $DM$ and the Auslander-Reiten translate of $M$ by $\tau M$ \cite{ARS}. For a complex $T$, we denote a shifted complex by $n$ degrees of $T$ by $T[n]$.\\

\section{Preliminaries}

In this section, we recall basic materials on group algebras and $\tau$-tilting theory.\\
\indent For a finite dimensional $k$-algebra $\Lambda$, let $\Lambda=B_1\oplus\cdots\oplus B_t$ be the decomposition of the $\Lambda$-bimodule $\Lambda$ into indecomposable $\Lambda$-subbimodules $B_i$. We say that this decomposition is a \textit{block decomposition of $\Lambda$} and each indecomposable direct summand $B_i$ is a \textit{block} of $\Lambda$. Each block is a $k$-algebra and the block decomposition induces the decomposition of the module category $\Lambda\textrm{-mod}\cong B_1\textrm{-mod}\times\cdots\times B_t\textrm{-mod}$.\\
\indent When $\Lambda:=kG$ for finite group $G$, there exists the unique block $B$ which contains the trivial $kG$-module $k_G$ in $B$-mod, that is, $Bk_G\neq0$. We call this block the \textit{principal block} and denote by $B_0(kG)$.\\

\begin{prop}[{\cite[Proposition 1.5.3]{S}}]\label{idem}
    For a finite dimensional $k$-algebra $\Lambda$ and  its block decomposition $\Lambda=B_1\oplus\cdots\oplus B_t$, there exist pairwise orthogonal primitive central idempotents $e_1,\ldots,e_t\in\Lambda$ such that $1=e_1+\cdots+e_t$ and $B_i\cong\Lambda e_i$.\\
\end{prop}

\indent We say that finite dimensional $k$-algebras $\Lambda$ and $\Gamma$ are \textit{derived equivalent} if their derived categories $D^b(\Lambda\textrm{-mod})$ and $D^b(\Gamma\textrm{-mod})$ are equivalent as triangulated categories.\\

\begin{dfn}
    Let $\Lambda$ be a finite dimensional $k$-algebra. A complex $T\in K^b(\Lambda\textrm{-}\mathrm{proj})$ is \textit{tilting} (resp. \textit{silting}) if $T$ satisfies the following conditions:
    \begin{enumerate}
        \setlength{\parskip}{0cm}
        \setlength{\itemsep}{0cm}
        \item $\mathrm{Hom}_{K^b(\Lambda\textrm{-}\mathrm{proj})}(T,T[i])=0$ for all integers $i\neq 0$ (resp. $i>0$),
        \item The full subcategory $\mathrm{add}\,T$ of $K^b(\Lambda\textrm{-}\mathrm{proj})$ consisting of all complexes isomorphic to direct sums of direct summands of $T$ generates $K^b(\Lambda\textrm{-}\mathrm{proj})$ as a triangulated category.\\
    \end{enumerate}
\end{dfn}

\begin{thm}[{\cite[Theorem 6.4]{R}}]\label{ric}
    Let $\Lambda$ and $\Gamma$ be finite dimensional $k$-algebras. The following conditions are equivalent:
    \begin{enumerate}
        \setlength{\parskip}{0cm}
        \setlength{\itemsep}{0cm}
        \item $\Lambda$ and $\Gamma$ are derived equivalent,
        \item There exists a tilting complex $T\in K^b(\Lambda\textrm{-}\mathrm{proj})$ such that $\Gamma\cong\mathrm{End}_{K^b(\Lambda\textrm{-}\mathrm{proj})}(T)^{\mathrm{op}}$.\\
    \end{enumerate}
\end{thm}

\indent We say that a finite dimensional $k$-algebra $\Lambda$ is \textit{symmetric} if $\Lambda$ is isomorphic to $D\Lambda$ as a $\Lambda$-bimodule. Group algebras and their blocks are symmetric. For a finite dimensional symmetric $k$-algebra $\Lambda$, a complex $T\in K^b(\Lambda\textrm{-}\mathrm{proj})$ is tilting if and only if it is silting \cite[Example 2.8]{AI}. We have an operation called a \textit{silting mutation}, which allows us to obtain various silting complexes from one silting complex \cite{AI}. Therefore, it is fundamental to study the structure of silting complexes when studying derived categories over group algebras and their blocks. However, it is difficult to determine the structure of all the silting complexes. Henceforth, we focus on two-term silting complexes, which have better properties.\\
\indent A complex $T=(\cdots\rightarrow T^{-1}\rightarrow T^0 \rightarrow T^1\rightarrow\cdots)\in K^b(\Lambda\textrm{-}\mathrm{proj})$ is \textit{two-term} if $T^i=0$ for all integers $i\neq -1,0$.\\

\begin{dfn}
    Let $\Lambda$ be a finite dimensional $k$-algebra. A module $M\in\Lambda\textrm{-}\mathrm{mod}$ is a \textit{support $\tau$-tilting module} if $M$ satisfies the following conditions:
    \begin{enumerate}
        \setlength{\parskip}{0cm}
        \setlength{\itemsep}{0cm}
        \item $M$ is $\tau$-rigid, that is, $\mathrm{Hom}_\Lambda(M,\tau M)=0$,
        \item There exists $P\in\Lambda\textrm{-}\mathrm{proj}$ such that $\mathrm{Hom}_{\Lambda}(P,M)=0$ and $|P|+|M|=|\Lambda|$.
    \end{enumerate}
    When we specify the projective $\Lambda$-module $P$ in (b), we write a support $\tau$-tilting module $M$ as a pair $(M,P)$.\\
\end{dfn}

\begin{prop}[{\cite[Theorem 2.7]{AIR}}]\label{order1}
    Let $\Lambda$ be a finite dimensional $k$-algebra. The set of isomorphism classes of support $\tau$-tilting $\Lambda$-modules is a partially ordered set with respect to the following relation:
    $$M\geq M'\Leftrightarrow there\;exists\;a\;surjective\;homomorphism\;from\;a\;finite\;direct\;sum\;of\;M\;to\;M'.$$
\end{prop}

\vspace{3mm}

\begin{prop}[{\cite[Theorem 2.11]{AI}}]\label{order2}
    Let $\Lambda$ be a finite dimensional $k$-algebra. The set of isomorphism classes of silting complexes in $K^b(\Lambda\textrm{-}\mathrm{proj})$ is a partially ordered set with respect to the following relation:
    $$T\geq T'\Leftrightarrow\mathrm{Hom}_{K^b(\Lambda\textrm{-}\mathrm{proj})}(T,T'[i])=0\;for\;all\;integers\;i>0.$$
\end{prop}

\vspace{3mm}

\indent We denote by $\mathrm{s}\tau\textrm{-}\mathrm{tilt}\,\Lambda$ the set of isomorphism classes of basic support $\tau$-tilting $\Lambda$-modules and by $2\textrm{-}\mathrm{silt}\,\Lambda$ the set of isomorphism classes of basic two-term silting complexes in $K^b(\Lambda\textrm{-}\mathrm{proj})$.\\

\begin{thm}[{\cite[Theorem 3.2]{AIR}}]\label{air}
    Let $\Lambda$ be a finite dimensional $k$-algebra. There exists an isomorphism as partially ordered sets
    $$
    \begin{tikzcd}[row sep = 1mm]
        \mathrm{s}\tau\textrm{-}\mathrm{tilt}\,\Lambda \rar[leftrightarrow, "\sim"] \dar[phantom, "\rotatebox{90}{$\in$}"] & 2\textrm{-}\mathrm{silt}\,\Lambda, \dar[phantom, "\rotatebox{90}{$\in$}"]\\
        (M,P) \rar[mapsto] & (P^M_1\oplus P\xrightarrow{(f\;0)}P^M_0), \\
        \mathrm{Cok}\,g \rar[mapsfrom] & (P^{-1}\xrightarrow{g}P^0), 
    \end{tikzcd}
    $$
    where $P^M_1\xrightarrow{f}P^M_0\rightarrow M\rightarrow0$ is the minimal projective presentation of $M$.\\
\end{thm}

\indent We recall the definition of functorially finite torsion classes and left finite semibricks. We call a subcategory $\mathcal{C}$ in $\Lambda\textrm{-}\mathrm{mod}$ a \textit{torsion class} if $\mathcal{C}$ is closed under taking factor modules and extensions. We say that a subcategory $\mathcal{C}$ in $\Lambda\textrm{-}\mathrm{mod}$ is \textit{functorially finite} if for any $M\in\Lambda\textrm{-}\mathrm{mod}$, there exist $X,\,Y\in\mathcal{C}$ and $f:M\rightarrow X,\,g:Y\rightarrow M\in\Lambda\textrm{-}\mathrm{mod}$ such that $-\circ f:\mathrm{Hom}_\Lambda(X,\mathcal{C})\rightarrow\mathrm{Hom}_\Lambda(M,\mathcal{C})$ and $g\circ -:\mathrm{Hom}_\Lambda(\mathcal{C},Y)\rightarrow\mathrm{Hom}_\Lambda(\mathcal{C},M)$ are both surjective.\\

\begin{dfn}
    Let $\Lambda$ be a finite dimensional $k$-algebra and $M\in\Lambda\textrm{-}\mathrm{mod}$. 
    \begin{enumerate}
        \setlength{\parskip}{0cm}
        \setlength{\itemsep}{0cm}
        \item We call $S$ a \textit{brick} if $\mathrm{End}_\Lambda(S)\cong k$.
        \item We call $S$ a \textit{semibrick} if $S$ is a direct sum of Hom-orthogonal bricks, that is, there exists a decomposition $S=S_1\oplus\cdots\oplus S_t$ into bricks such that $\mathrm{Hom}_\Lambda(S_i,S_j)=0$ for $i\neq j$.
        \item We say that a semibrick $S$ is \textit{left finite} if the smallest torsion class containing $S$ is functorially finite.
        \item We denote the set of isomorphism classes of (resp. left finite) semibricks over $\Lambda$ by $\mathrm{sbrick}\,\Lambda$ (resp. $\mathrm{f_L\textrm{-}sbrick}\,\Lambda$).\\
    \end{enumerate}
\end{dfn}

\begin{thm}[{\cite[Theorem 2.3]{A}}]\label{sbrick}
    There exists a bijection
    \begin{equation}
        \mathrm{s}\tau\textrm{-}\mathrm{tilt}\,\Lambda\xrightarrow{\sim}\mathrm{f_L\textrm{-}sbrick}\,\Lambda, \label{eqsbrick}
    \end{equation}
     which sends $M\in\mathrm{s}\tau\textrm{-}\mathrm{tilt}\,\Lambda$ to $M/\sum_{f\in\mathrm{Rad}_\Lambda(M,M)}\mathrm{Im}\,f\in\mathrm{f_L\textrm{-}sbrick}\,\Lambda$.\\
\end{thm}

\section{Main results}
For a finite dimensional $k$-algebra $\Lambda$, we denote by $Z(\Lambda)$ the center of $\Lambda$ and by $J(\Lambda)$ the Jacobson radical of $\Lambda$. In \cite{EJR}, Eisele, Janssens and Raedschelders showed the following reduction theorem for support $\tau$-tilting modules by giving an isomorphism in terms of two-term silting complexes.\\

\begin{thm}[{\cite[Theorem 11]{EJR}}]\label{reduction}
    Let $\Lambda$ be a finite dimensional $k$-algebra. For any ideal $I\subset (Z(\Lambda)\cap J(\Lambda))\cdot\Lambda$ of $\Lambda$, we have an isomorphism
    \begin{equation}
        \mathrm{s}\tau\textrm{-}\mathrm{tilt}\,\Lambda\cong\mathrm{s}\tau\textrm{-}\mathrm{tilt}\,\Lambda/I \label{eqreduction}
    \end{equation}
    as partially ordered sets.\\
\end{thm}

We give the explicit correspondence for support $\tau$-tilting modules in \eqref{eqreduction} to formulate the reduction theorem for left finite semibricks and examine the correspondence at the level of blocks.\\

\begin{cor}\label{module}
    In the setting of Theorem $\ref{reduction}$, the bijection sends $M\in\mathrm{s}\tau\textrm{-}\mathrm{tilt}\,\Lambda$ to $M/IM\in\mathrm{s}\tau\textrm{-}\mathrm{tilt}\,\Lambda/I$.
\end{cor}
\begin{proof}
    As in the proof in \cite[Theorem 11]{EJR}, the functor $\widetilde{(-)}:=\Lambda/I\otimes_\Lambda -:\Lambda\textrm{-mod}\rightarrow(\Lambda/I)\textrm{-mod}$ preserves minimal projective presentations of support $\tau$-tilting modules. More precisely, if $M\in\mathrm{s}\tau\textrm{-}\mathrm{tilt}\,\Lambda$ corresponds to $M'\in\mathrm{s}\tau\textrm{-}\mathrm{tilt}(\Lambda/I)$ under the bijection and the exact sequence $P_1\xrightarrow{f} P_0\rightarrow M\rightarrow 0$ is the minimal projective presentation of $M$, then the exact sequence $\widetilde{P_1}\xrightarrow{\widetilde{f}} \widetilde{P_0}\rightarrow M'\rightarrow 0$ is also the minimal projective presentation of $M'$. Since the functor $\widetilde{(-)}$ is right exact, 
    $$M'\cong \mathrm{Cok}\,\widetilde{f} \cong\widetilde{\mathrm{Cok}\,f}\cong\widetilde{M}\cong M/IM.$$
\end{proof}

\vspace{3mm}

\begin{cor}\label{redsbrick}
    In the setting of Theorem \ref{reduction}, the natural embedding $\Lambda/I\textrm{-}\mathrm{mod}\rightarrow\Lambda\textrm{-}\mathrm{mod}$ induces a bijection
    \begin{equation}
        \mathrm{sbrick}\,\Lambda/I\xrightarrow{\sim}\mathrm{sbrick}\,\Lambda. \label{eqredsbrick}
    \end{equation}
    Moreover, this bijection preserves left finiteness and the following diagram commutes:
    $$
    \begin{tikzcd}
        \mathrm{s}\tau\textrm{-}\mathrm{tilt}\,\Lambda \ar[r,"\eqref{eqreduction}"] \ar[d,"\eqref{eqsbrick}"] & \mathrm{s}\tau\textrm{-}\mathrm{tilt}\,\Lambda/I, \ar[d,"\eqref{eqsbrick}"] \\
        \mathrm{f_L}\textrm{-}\mathrm{sbrick}\,\Lambda & \mathrm{f_L}\textrm{-}\mathrm{sbrick}\,\Lambda/I \ar[l,"\eqref{eqredsbrick}"].
    \end{tikzcd}
    $$
\end{cor}
\begin{proof}
    The natural embedding $\Lambda/I\textrm{-}\mathrm{mod}\rightarrow\Lambda\textrm{-}\mathrm{mod}$ induces an injection $\mathrm{sbrick}\,\Lambda/I\hookrightarrow\mathrm{sbrick}\,\Lambda$. Hence, for the first assertion, it suffices to show that $IS=0$ for any $S\in\mathrm{sbrick}\,\Lambda$. Take any $x\in I$. We can define a homomorphism $f\in\mathrm{End}_\Lambda(S)$ by $f(s):=xs\;(s\in S)$ because $x\in I\subset Z(\Lambda)$. Since $x\in I\subset J(\Lambda)$, $x$ is nilpotent, so is $f$. However, since $\mathrm{End}_\Lambda(S)$ is isomorphic to a direct product of copies of $k$ as a $k$-algebra, we have $f=0$. Thus, it follows that $xS=0$ for every $x\in I$, which complete the proof of the first assertion.\\
    \indent We can pick $a_1,\ldots,a_t\in Z(\Lambda)\cap J(\Lambda)$ such that $I=a_1\Lambda+\cdots+a_t\Lambda$. Take arbitrary $M\in\mathrm{s}\tau\textrm{-}\mathrm{tilt}\,\Lambda$ and denote the natural surjection $M\rightarrow M/IM$ by $\overline{(-)}$. Let $\Gamma:=\mathrm{End}_{\Lambda}(M)$ and $\Theta:=\mathrm{End}_{\Lambda/I}(M/IM)\cong\mathrm{End}_{\Lambda}(M/IM)$. For any $f\in\Gamma$, we have $f(IM)=If(M)\subset IM$. Hence, we can define a $k$-algebra homomorphism $\varphi:\Gamma\rightarrow\Theta$ by $\varphi(f)(\overline{m}):=\overline{f(m)}\;(f\in\Gamma,m\in M)$, which induces the following commutative diagram:
    $$
    \begin{tikzcd}
        0 \ar[r] & IM \ar[r] & M \ar[r,"\overline{(-)}"] \ar[d,"f"] & M/IM \ar[r] \ar[d,"\varphi(f)"] & 0, \\
        0 \ar[r] & IM \ar[r] & M \ar[r,"\overline{(-)}"] & M/IM \ar[r] & 0. 
    \end{tikzcd}
    $$
    \indent First, we show that $IM\in\mathrm{Fac}\,M$ and $IM\subset\sum_{f\in J(\Gamma)}\mathrm{Im}\,f$, where $\mathrm{Fac}\,M$ denotes a full subcategory consisting of factor modules of direct sums of copies of $M$. For integers $1\leq i\leq t$, we can define a homomorphism $f_i\in\Gamma$ by $f_i(m):=a_im\;(m\in M)$ since $a_i\in Z(\Lambda)$. Then a homomorphism $(f_i)_{1\leq i\leq t}:M^{\oplus t}\rightarrow IM$ is well-defined and surjective, which implies $IM\in\mathrm{Fac}\,M$. Since $a_i\in J(\Lambda)$, there exists some integer $l>0$ such that $(f_i)^l=0$. Thus, we have $f_i\in J(\Gamma)$ and $IM=\sum_{1\leq i\leq t}f_i(M)\subset \sum_{f\in J(\Gamma)}\mathrm{Im}\,f$.\\
    \indent Next, we show that $\varphi$ is surjective. By applying the functor $\mathrm{Hom}_{\Lambda}(M,-)$ to the short exact sequence
    $$0\rightarrow IM\rightarrow M\xrightarrow{\overline{(-)}} M/IM\rightarrow 0,$$
    we get the exact sequence
    $$\mathrm{Hom}_{\Lambda}(M,M)\rightarrow\mathrm{Hom}_{\Lambda}(M,M/IM)\rightarrow\mathrm{Ext}^1_{\Lambda}(M,IM).$$
    By $\tau$-rigidity of $M$ and \cite[Proposition 5.8]{AS}, it follows that $\mathrm{Ext}^1_{\Lambda}(M,\mathrm{Fac}\,M)=0$. Hence, $\mathrm{Ext}^1_{\Lambda}(M,IM)=0$ and the map $\mathrm{Hom}_{\Lambda}(M,M)\rightarrow\mathrm{Hom}_{\Lambda}(M,M/IM)$ is surjective. Therefore, $\varphi$ is also surjective by the commutative diagram above.\\
    \indent Since $\mathrm{Ker}\,\varphi=\{f\in \Gamma\mid f(M)\subset IM\}$ and $I$ is nilpotent, $\mathrm{Ker}\,\varphi$ is nilpotent, which implies $\mathrm{Ker}\,\varphi\subset J(\Gamma)$. Hence, we have $ J(\Gamma)/\mathrm{Ker}\,\varphi=J(\Gamma/\mathrm{Ker}\,\varphi)\cong J(\Theta)$ and $\varphi(J(\Gamma))=J(\Theta)$. By considering the following commutative diagram for any $m\in M$;
    $$
    \begin{tikzcd}
        J(\Gamma) \ar[r,"\mathrm{ev}_{m}"] \ar[d, two heads, "\varphi"] & M, \ar[d, two heads, "\overline{(-)}"] \\
        J(\Theta) \ar[r,"\mathrm{ev}_{\overline{m}}"] & M/IM,
    \end{tikzcd}
    $$
    where $\mathrm{ev}_m(f):=f(m)\;(f\in\Gamma)$ and $\mathrm{ev}_{\overline{m}}(f'):=f'(\overline{m})\;(f'\in\Theta)$, it follows that
    $$\overline{\sum_{f\in J(\Gamma)}\mathrm{Im}\,f}=\sum_{f'\in J(\Theta)} \mathrm{Im}\,f',$$
    which implies
    $$\frac{M/IM}{\sum_{f'\in J(\Theta)} \mathrm{Im}\,f'}=\frac{M/IM}{(\sum_{f\in J(\Gamma)}\mathrm{Im}\,f)/IM}\cong \frac{M}{\sum_{f\in J(\Gamma)}\mathrm{Im}\,f}$$
    as $\Lambda$-modules. Therefore, the second assertion holds because the following diagram commutes:
    $$
    \begin{tikzcd}
        M \dar[mapsto] \rar[phantom, near end,"\in"] & \mathrm{s}\tau\textrm{-}\mathrm{tilt}\,\Lambda \dar["\eqref{eqsbrick}"] \rar["\eqref{eqreduction}"] & \mathrm{s}\tau\textrm{-}\mathrm{tilt}\,\Lambda/I \dar["\eqref{eqsbrick}"] \rar[phantom, near start, "\ni"] & M/IM, \dar[mapsto] \\
        M/\sum_{f\in J(\Gamma)}\mathrm{Im}\,f \rar[phantom, "\in"] & \mathrm{f_L\textrm{-}sbrick}\,\Lambda & \mathrm{f_L\textrm{-}sbrick}\,\Lambda/I \lar["\eqref{eqredsbrick}"] \rar[phantom, "\ni"] & (M/IM)/\sum_{f'\in J(\Theta)} \mathrm{Im}\,f', \\[-5mm]
         & \Lambda\textrm{-}\mathrm{mod} \uar[phantom,"\cap"] & \Lambda/I\textrm{-}\mathrm{mod}. \uar[phantom,"\cap"] \lar[hook'] & 
    \end{tikzcd}
    $$
\end{proof}

\vspace{3mm}

\begin{cor}\label{block}
    In the setting of Theorem $\ref{reduction}$, the following assertions hold:
    \begin{enumerate}
        \setlength{\parskip}{0cm}
        \setlength{\itemsep}{0cm}
        \item The natural surjection $\Lambda\rightarrow\Lambda/I$ induces a bijection between the set of blocks of $\Lambda$ and that of $\Lambda/I$.
        \item If a block $B$ of $\Lambda$ corresponds to a block $B'$ of $\Lambda/I$ under the bijection in (a), then the isomorphism in Theorem $\ref{reduction}$ induces an isomorphism as partially ordered sets
        \begin{equation}
            \mathrm{s}\tau\textrm{-}\mathrm{tilt}\,B\cong\mathrm{s}\tau\textrm{-}\mathrm{tilt}\,B'. \label{eqblockred}
        \end{equation}
        \item If a block $B$ of $\Lambda$ corresponds to a block $B'$ of $\Lambda/I$ under the bijection in (a), then the natural embedding $\Lambda/I\textrm{-}\mathrm{mod}\rightarrow\Lambda\textrm{-}\mathrm{mod}$ induces a bijection
        \begin{equation}
            \mathrm{sbrick}\,B'\xrightarrow{\sim}\mathrm{sbrick}\,B. \label{eqblockredsbrick}
        \end{equation}
        \item If a block $B$ of $\Lambda$ corresponds to a block $B'$ of $\Lambda/I$ under the bijection in (a), then the bijection in (c) preserves left finiteness and the following diagram commutes:
        $$
        \begin{tikzcd}
            \mathrm{s}\tau\textrm{-}\mathrm{tilt}\,B \ar[r,"\eqref{eqblockred}"] \ar[d,"\eqref{eqsbrick}"] & \mathrm{s}\tau\textrm{-}\mathrm{tilt}\,B', \ar[d,"\eqref{eqsbrick}"] \\
            \mathrm{f_L}\textrm{-}\mathrm{sbrick}\,B & \mathrm{f_L}\textrm{-}\mathrm{sbrick}\,B' \ar[l,"\eqref{eqblockredsbrick}"].
        \end{tikzcd}
        $$
    \end{enumerate}
\end{cor}
\begin{proof}
    It suffices to show that the assertions hold when $I=a\Lambda$ for any $a\in Z(\Lambda)\cap J(\Lambda)$. Let $\Lambda=B_1\oplus\cdots\oplus B_t$ be the block decomposition of $\Lambda$ and we set $a=(a_1,\cdots,a_t)\in B_1\oplus\cdots\oplus B_t$. Since $Z(\Lambda)=Z(B_1)\oplus\cdots\oplus Z(B_t)$ and $J(\Lambda)=J(B_1)\oplus\cdots\oplus J(B_t)$, we have $a_i\in Z(B_i)\cap J(B_i)$ for all $1\leq i\leq t$. It follows that $\Lambda/a\Lambda=B_1/a_1B_1\oplus\cdots\oplus B_t/a_tB_t$ and $B_i/a_iB_i\neq 0$ for all integers $1\leq i\leq t$.\\
    \indent To prove (a), it suffices to show that $B_i/a_iB_i$ does not have non-trivial central idempotents for all $1\leq i\leq t$. Let $\varepsilon$ be a central idempotent of $B_i/a_iB_i$. Since $a_iB_i\subset J(B_i)$, there exists an idempotent $e\in B_i$ such that $\varepsilon=e+a_iB_i$ \cite[Propositon 1.5.7]{S}. We have $$(1-\varepsilon)(B_i/a_iB_i)\varepsilon=(1-\varepsilon)\varepsilon(B_i/a_iB_i)=0,$$
    so $(1-e)B_ie\subset a_iB_i$ holds. Since $a_i\in Z(B_i)$, we have
    $$(1-e)B_ie=(1-e)^2B_ie^2\subset(1-e)a_iB_ie=a_i(1-e)B_ie.$$
    We have an integer $N$ such that $a_i^N=0$ since $a_i\in J(B_i)$. Thus, we have
    $$(1-e)B_ie\subset a_i(1-e)B_ie\subset a_i^2(1-e)B_ie\subset\cdots\subset a_i^N(1-e)B_ie=0.$$
    It follows that $(1-e)B_ie=0$. In the same way, we also have $eB_i(1-e)=0$. Therefore, we have\\
    $$B_i=eB_ie+(1-e)B_ie+eB_i(1-e)+(1-e)B_i(1-e)=eB_ie+(1-e)B_i(1-e),$$
    hence $e$ is a central idempotent of $B_i$. Since any block has no nontrivial central idempotents, we have $e=0$ or $1$. Thus, it follows that $\varepsilon=0$ or $1$, which shows (a).\\
    \indent (b), (c) and (d) follow from Theorem $\ref{reduction}$ and Corollary \ref{redsbrick} since the block $B_i$ of $\Lambda$ corresponds to the block $B_i/a_iB_i$ of $\Lambda/a\Lambda$ and $a_i\in Z(B_i)\cap J(B_i)$.
\end{proof}

\vspace{3mm}

\indent In the rest of this section, we assume $p:=\mathrm{char}\,k>0$ and let $G$ be a finite group, $Z(G)$ the center of $G$, $1=Z_0\leq Z_1\leq\cdots$ the upper central series of $G$, and $H(G)$ the final stable term of the upper central series of $G$. We prepare the required propositions about group algebras and apply the above reduction theorem to the case of group algebras to obtain our main results.\\

\begin{prop}\label{mainprop}
    For any normal subgroup $N$ of $G$, the kernel of the natural surjection $\pi:kG\rightarrow k(G/N)$ of group algebras is $\{1-n\mid n\in N\}\cdot kG$. Moreover, if $N$ is a $p$-group, then the kernel of $\pi$ is $J(kN)\cdot kG$.
\end{prop}
\begin{proof}
    Let $\{g_1=1,\ldots,g_m\}$ be the set of representatives of right cosets of $N$ in $G$.\\
    \indent Let $\mathcal{B}:=\{g_i-ng_i\mid 1\leq i\leq m,\,n\in N-\{1\}\}\subset kG$. Every element $g\in G-\{g_1,\ldots,g_m\}$ appears exactly once in the terms of the elements in $\mathcal{B}$ because it can be written uniquely in the form of $g=ng_i\;(1\leq i\leq m,\,n\in N)$. Hence, the set  $\mathcal{B}$ is linearly independent. Since
    $\mathcal{B}\subset\{1-n\mid n\in N\}\cdot kG\subset\mathrm{Ker}\,\pi$,
    we have $|\mathcal{B}|\leq \mathrm{dim}_k\{1-n\mid n\in N\}\cdot kG\leq \mathrm{dim}_k\mathrm{Ker}\,\pi$. However, 
    $$\mathrm{dim}_k\mathrm{Ker}\,\pi=\mathrm{dim}_k kG-\mathrm{dim}_k k(G/N)=m|N|-m=|\mathcal{B}|,$$
    so it follows that $\mathrm{Ker}\,\pi=\{1-n\mid n\in N\}\cdot kG$.\\
    \indent If $N$ is a $p$-group, then the set $\{1-n\mid n\in N-\{1\}\}\subset kN$ is a $k$-basis of $J(kN)$ by \cite[Corollary in the section 4.7]{P}. Therefore, we have $\mathrm{Ker}\,\pi=J(kN)\cdot kG$.
\end{proof}

\vspace{3mm}

\begin{prop}\label{mainprop'}
    For any normal $p'$-subgroup $N$ of $G$, the restriction of the natural surjection $\pi:kG\rightarrow k(G/N)$ to $B_0(kG)$ induces an isomorphism $\pi|_{B_0(kG)}:B_0(kG)\xrightarrow{\sim}B_0(k(G/N))$ of $k$-algebras.
\end{prop}
\begin{proof}
    Let $\beta:=|N|^{-1}\sum_{n\in N}n\in kG$. Clearly, $\beta$ is a central idempotent of $kG$ and $B_0(kG)$ is contained in $kG\beta$. Hence, it suffices to show that the restriction $\pi|_{kG\beta}:kG\beta\rightarrow k(G/N)$ is an isomorphism.\\
    \indent Since $\pi(\beta)=1$, $\pi|_{kG\beta}$ is surjective. By Proposition \ref{mainprop}, we have
    $$\mathrm{Ker}\,\pi|_{kG\beta}=(\mathrm{Ker}\,\pi)\beta=\{1-n\mid n\in N\}\cdot kG\beta=0,$$
    where the last equation follows because $(1-n)g\beta=(\beta-n\beta)g=0$ for any $n\in N$ and $g\in G$. Thus, $\pi|_{kG\beta}$ is injective.  Therefore, $\pi|_{kG\beta}$ is an isomorphism.
\end{proof}

\vspace{3mm}

\indent We are now ready to state the `quotient reduction' theorem.\\

\begin{thm}\label{mainthm}
    For any $p$-subgroup $N$ of $Z(G)$, the following assertions hold:
    \begin{enumerate}
        \setlength{\parskip}{0cm}
        \setlength{\itemsep}{0cm}
        \item There exists an isomorphism as partially ordered sets
        \begin{equation}
            \mathrm{s}\tau\textrm{-}\mathrm{tilt}\,kG\cong\mathrm{s}\tau\textrm{-}\mathrm{tilt}\,k(G/N), \label{eqgpred1}
        \end{equation}
        which sends $M\in \mathrm{s}\tau\textrm{-}\mathrm{tilt}\,kG$ to $M/J(kN)M\in \mathrm{s}\tau\textrm{-}\mathrm{tilt}\,k(G/N)$.
        \item The natural embedding $k(G/N)\textrm{-}\mathrm{mod}\rightarrow kG\textrm{-}\mathrm{mod}$ induces a bijection
        \begin{equation}
            \mathrm{sbrick}\,k(G/N)\xrightarrow{\sim}\mathrm{sbrick}\,kG. \label{eqgpredsbrick1}
        \end{equation}
        \item The bijection in (b) preserves left finiteness and the following diagram commutes:
        $$
        \begin{tikzcd}
            \mathrm{s}\tau\textrm{-}\mathrm{tilt}\,kG \ar[r,"\eqref{eqgpred1}"] \ar[d,"\eqref{eqsbrick}"] & \mathrm{s}\tau\textrm{-}\mathrm{tilt}\,k(G/N), \ar[d,"\eqref{eqsbrick}"] \\
            \mathrm{f_L}\textrm{-}\mathrm{sbrick}\,kG & \mathrm{f_L}\textrm{-}\mathrm{sbrick}\,k(G/N) \ar[l,"\eqref{eqgpredsbrick1}"].
        \end{tikzcd}
        $$
        \item The natural surjection $kG\rightarrow k(G/N)$ induces a bijection between the set of blocks of $kG$ and that of $k(G/N)$. In particular, this bijection sends $B_0(kG)$ to $B_0(k(G/N))$.
        \item If a block $B$ of $kG$ corresponds to a block $B'$ of $k(G/N)$ under the bijection in (d), then the isomorphism in (a) induces an isomorphism as partially ordered sets
        \begin{equation}
            \mathrm{s}\tau\textrm{-}\mathrm{tilt}\,B\cong\mathrm{s}\tau\textrm{-}\mathrm{tilt}\,B'. \label{eqgpblockred1}
        \end{equation}
        \item If a block $B$ of $kG$ corresponds to a block $B'$ of $k(G/N)$ under the bijection in (d), then the natural embedding $k(G/N)\textrm{-}\mathrm{mod}\rightarrow kG\textrm{-}\mathrm{mod}$ induces a bijection
        \begin{equation}
            \mathrm{sbrick}\,B'\xrightarrow{\sim}\mathrm{sbrick}\,B. \label{eqgpblockredsbrick1}
        \end{equation}
        \item If a block $B$ of $kG$ corresponds to a block $B'$ of $k(G/N)$ under the bijection in (d), then the bijection in (f) preserves left finiteness and the following diagram commutes:
        $$
        \begin{tikzcd}
            \mathrm{s}\tau\textrm{-}\mathrm{tilt}\,B \ar[r,"\eqref{eqgpblockred1}"] \ar[d,"\eqref{eqsbrick}"] & \mathrm{s}\tau\textrm{-}\mathrm{tilt}\,B', \ar[d,"\eqref{eqsbrick}"] \\
            \mathrm{f_L}\textrm{-}\mathrm{sbrick}\,B & \mathrm{f_L}\textrm{-}\mathrm{sbrick}\,B' \ar[l,"\eqref{eqgpblockredsbrick1}"].
        \end{tikzcd}
        $$
    \end{enumerate}
\end{thm}
\begin{proof}
    By Theorem \ref{reduction}, Corollaries \ref{module}, \ref{redsbrick}, \ref{block} and Proposition \ref{mainprop}, it is sufficient to show that $J(kN)\subset Z(kG)\cap J(kG)$. Since $N\subset Z(G)$, we have $J(kN)\subset kN\subset Z(kG)$. By \cite[Corollary in the section 4.7]{P}, the set $\{1-n\mid n\in N-\{1\}\}$ is a $k$-basis of $J(kN)$. For any $g\in G$ and $n\in N-\{1\}$, there exists some integer $l$ such that $n^{p^l}=1$, so we have
    $$(g(1-n))^{p^l}=g^{p^l}(1-n)^{p^l}=g^{p^l}(1-n^{p^l})=0$$
    and it follows that $J(kN)\subset J(kG)$.
\end{proof}

\vspace{3mm}

\begin{cor}\label{maincor}
    Let $N$ be the $p$-Sylow subgroup of $H(G)$. Then the following assertions hold:
    \begin{enumerate}
        \setlength{\parskip}{0cm}
        \setlength{\itemsep}{0cm}
        \item There exists an isomorphism as partially ordered sets
        \begin{equation}
            \mathrm{s}\tau\textrm{-}\mathrm{tilt}\,kG\cong\mathrm{s}\tau\textrm{-}\mathrm{tilt}\,k(G/N), \label{eqgpred2}
        \end{equation}
        which sends $M\in \mathrm{s}\tau\textrm{-}\mathrm{tilt}\,kG$ to $M/J(kN)M\in \mathrm{s}\tau\textrm{-}\mathrm{tilt}\,k(G/N)$.
        \item The natural embedding $k(G/N)\textrm{-}\mathrm{mod}\rightarrow kG\textrm{-}\mathrm{mod}$ induces a bijection
        \begin{equation}
            \mathrm{sbrick}\,k(G/N)\xrightarrow{\sim}\mathrm{sbrick}\,kG. \label{eqgpredsbrick2}
        \end{equation}
        \item The bijection in (b) preserves left finiteness and the following diagram commutes:
        $$
        \begin{tikzcd}
            \mathrm{s}\tau\textrm{-}\mathrm{tilt}\,kG \ar[r,"\eqref{eqgpred2}"] \ar[d,"\eqref{eqsbrick}"] & \mathrm{s}\tau\textrm{-}\mathrm{tilt}\,k(G/N), \ar[d,"\eqref{eqsbrick}"] \\
            \mathrm{f_L}\textrm{-}\mathrm{sbrick}\,kG & \mathrm{f_L}\textrm{-}\mathrm{sbrick}\,k(G/N) \ar[l,"\eqref{eqgpredsbrick2}"].
        \end{tikzcd}
        $$
        \item The natural surjection $kG\rightarrow k(G/N)$ induces a bijection between the set of blocks of $kG$ and that of $k(G/N)$. In particular, this bijection sends $B_0(kG)$ to $B_0(k(G/N))$.
        \item If a block $B$ of $kG$ corresponds to a block $B'$ of $k(G/N)$ under the bijection in (d), then the isomorphism in (a) induces an isomorphism as partially ordered sets
        \begin{equation}
            \mathrm{s}\tau\textrm{-}\mathrm{tilt}\,B\cong\mathrm{s}\tau\textrm{-}\mathrm{tilt}\,B'. \label{eqgpblockred2}
        \end{equation}
        \item If a block $B$ of $kG$ corresponds to a block $B'$ of $k(G/N)$ under the bijection in (d), then the natural embedding $k(G/N)\textrm{-}\mathrm{mod}\rightarrow kG\textrm{-}\mathrm{mod}$ induces a bijection
        \begin{equation}
            \mathrm{sbrick}\,B'\xrightarrow{\sim}\mathrm{sbrick}\,B. \label{eqgpblockredsbrick2}
        \end{equation}
        \item If a block $B$ of $kG$ corresponds to a block $B'$ of $k(G/N)$ under the bijection in (d), then the bijection in (f) preserves left finiteness and the following diagram commutes:
        $$
        \begin{tikzcd}
            \mathrm{s}\tau\textrm{-}\mathrm{tilt}\,B \ar[r,"\eqref{eqgpblockred2}"] \ar[d,"\eqref{eqsbrick}"] & \mathrm{s}\tau\textrm{-}\mathrm{tilt}\,B', \ar[d,"\eqref{eqsbrick}"] \\
            \mathrm{f_L}\textrm{-}\mathrm{sbrick}\,B & \mathrm{f_L}\textrm{-}\mathrm{sbrick}\,B' \ar[l,"\eqref{eqgpblockredsbrick2}"].
        \end{tikzcd}
        $$
    \end{enumerate}
\end{cor}
\begin{proof}
    Note that $Z_i$ is a characteristic subgroup of $G$ for any integer $i\geq 0$. In particular, $H(G)$ is a characteristic subgroup of $G$. Since $H(G)$ is a finite nilpotent group, $N$ is the unique $p$-Sylow subgroup of $H(G)$ and hence $N$ is a characteristic subgroup of $H(G)$. Thus, $N$ is a characteristic subgroup of $G$. Therefore, $N\cap Z_i$ is a characteristic subgroup of $G$ for any integer $i\geq0$.\\
    \indent The factor group series
    $$G\twoheadrightarrow G/(N\cap Z_1)\twoheadrightarrow G/(N\cap Z_2)\twoheadrightarrow\cdots\twoheadrightarrow G/(N\cap H(G))=G/N$$
    induces the quotient group algebra series
    $$kG\twoheadrightarrow k(G/(N\cap Z_1))\twoheadrightarrow k(G/(N\cap Z_2))\twoheadrightarrow\cdots\twoheadrightarrow k(G/N).$$
    Thus, it suffices to show $(N\cap Z_i)/(N\cap Z_{i-1})\leq Z(G/(N\cap Z_{i-1}))$ for all integers $i>0$ because, if true, we can apply Theorem \ref{mainthm} to the group algebras $k(G/(N\cap Z_{i-1}))\twoheadrightarrow k(G/(N\cap Z_i))$ inductively. It is enough to show that $g^{-1}n^{-1}gn\in N\cap Z_{i-1}$ for any $g\in G$ and $n\in N\cap Z_i$. We have $g^{-1}n^{-1}gn\in Z_{i-1}$ because $Z_i/Z_{i-1}=Z(G/Z_{i-1})$. Since $N$ is normal in $G$, it follows that $(g^{-1}n^{-1}g)n\in N$. Therefore, we have $g^{-1}n^{-1}gn\in N\cap Z_{i-1}$. The proof is now complete.
\end{proof}

\vspace{3mm}

\indent Considering the correspondence over the principal blocks, we can relax the assumption in Theorem \ref{mainthm} or Corollary \ref{maincor}. We find that central extensions of finite groups have no effect on the structure of support $\tau$-tilting modules or left finite semibricks.\\

\begin{thm}\label{genmainthm}
    Let $N$ be the $p$-Sylow subgroup of $Z(G)$. Then the following assertions hold:
    \begin{enumerate}
        \setlength{\parskip}{0cm}
        \setlength{\itemsep}{0cm}
        \item There exists an isomorphism as partially ordered sets
        \begin{equation}
            \mathrm{s}\tau\textrm{-}\mathrm{tilt}\,B_0(kG)\cong\mathrm{s}\tau\textrm{-}\mathrm{tilt}\,B_0(k(G/Z(G))), \label{eqgpred3}
        \end{equation}
        which sends $M\in \mathrm{s}\tau\textrm{-}\mathrm{tilt}\,B_0(kG)$ to $M/J(kN)M\in \mathrm{s}\tau\textrm{-}\mathrm{tilt}\,B_0(k(G/Z(G)))$.
        \item The natural embedding $k(G/Z(G))\textrm{-}\mathrm{mod}\rightarrow kG\textrm{-}\mathrm{mod}$ induces a bijection
        \begin{equation}
            \mathrm{sbrick}\,B_0(k(G/Z(G)))\xrightarrow{\sim}\mathrm{sbrick}\,B_0(kG). \label{eqgpredsbrick3}
        \end{equation}
        \item The bijection in (b) preserves left finiteness and the following diagram commutes:
        $$
        \begin{tikzcd}
            \mathrm{s}\tau\textrm{-}\mathrm{tilt}\,B_0(kG) \ar[r,"\eqref{eqgpred3}"] \ar[d,"\eqref{eqsbrick}"] & \mathrm{s}\tau\textrm{-}\mathrm{tilt}\,B_0(k(G/Z(G))), \ar[d,"\eqref{eqsbrick}"] \\
            \mathrm{f_L}\textrm{-}\mathrm{sbrick}\,B_0(kG) & \mathrm{f_L}\textrm{-}\mathrm{sbrick}\,B_0(k(G/Z(G))) \ar[l,"\eqref{eqgpredsbrick3}"].
        \end{tikzcd}
        $$
    \end{enumerate}
\end{thm}
\begin{proof}
    Since $Z(G)/N$ is a normal $p'$-subgroup of $G/N$ and $(G/N)/(Z(G)/N)\cong G/Z(G)$, the natural surjection $k(G/N)\rightarrow k(G/Z(G))$ induces an isomorphism $B_0(k(G/N))\cong B_0(k(G/Z(G)))$ of $k$-algebras by Proposition \ref{mainprop'}. By this isomorphism and Theorem \ref{mainthm}, we have the assertions.
\end{proof}

\vspace{3mm}

\begin{cor}\label{genmaincor}
    Let $N$ be the $p$-Sylow subgroup of $H(G)$. Then the following assertions hold:
    \begin{enumerate}
        \setlength{\parskip}{0cm}
        \setlength{\itemsep}{0cm}
        \item There exists an isomorphism as partially ordered sets
        \begin{equation}
            \mathrm{s}\tau\textrm{-}\mathrm{tilt}\,B_0(kG)\cong\mathrm{s}\tau\textrm{-}\mathrm{tilt}\,B_0(k(G/H(G))), \label{eqgpred4}
        \end{equation}
        which sends $M\in \mathrm{s}\tau\textrm{-}\mathrm{tilt}\,B_0(kG)$ to $M/J(kN)M\in \mathrm{s}\tau\textrm{-}\mathrm{tilt}\,B_0(k(G/H(G)))$.
        \item The natural embedding $k(G/H(G))\textrm{-}\mathrm{mod}\rightarrow kG\textrm{-}\mathrm{mod}$ induces a bijection
        \begin{equation}
            \mathrm{sbrick}\,B_0(k(G/H(G)))\xrightarrow{\sim}\mathrm{sbrick}\,B_0(kG). \label{eqgpredsbrick4}
        \end{equation}
        \item The bijection in (b) preserves left finiteness and the following diagram commutes:
        $$
        \begin{tikzcd}
            \mathrm{s}\tau\textrm{-}\mathrm{tilt}\,B_0(kG) \ar[r,"\eqref{eqgpred4}"] \ar[d,"\eqref{eqsbrick}"] & \mathrm{s}\tau\textrm{-}\mathrm{tilt}\,B_0(k(G/H(G))), \ar[d,"\eqref{eqsbrick}"] \\
            \mathrm{f_L}\textrm{-}\mathrm{sbrick}\,B_0(kG) & \mathrm{f_L}\textrm{-}\mathrm{sbrick}\,B_0(k(G/H(G))) \ar[l,"\eqref{eqgpredsbrick4}"].
        \end{tikzcd}
        $$
    \end{enumerate}
\end{cor}
\begin{proof}
    In the same way as in the proof of Corollary \ref{maincor}, we can apply Theorem \ref{genmainthm} to the group algebras $k(G/Z_{i-1})\twoheadrightarrow k(G/Z_i)$ inductively and prove the assertions.
\end{proof}

\vspace{3mm}

\section{Examples}

(a) We assume $\mathrm{char}\,k=2$ and let $D_{2n}=\left<a,b\mid a^n=b^2=1,b^{-1}ab=a^{-1}\right>$. If $n$ is even, then we have $Z(D_{2n})=\{1,a^{n/2}\}$, so $D_{2n}/Z(D_{2n})=D_{2\cdot (n/2)}$. Hence, in order to determine the structure of $\mathrm{s}\tau\textrm{-}\mathrm{tilt}\,kD_{2n}$, it is sufficient to consider the case $n$ is odd by Theorem $\ref{mainthm}$. When $n$ is odd, the block decomposition of $kD_{2n}$ is
$$kD_{2n}=B_0\times B_1\times \cdots\times B_{(n-1)/2},$$
where $B_0\cong k[x]/(x^2)$ and $B_1, \ldots, B_{(n-1)/2}\cong M_2(k)$ (here $M_2(k)$ denotes the $2\times 2$ matrix algebra over $k$). Therefore, all the support $\tau$-tilting $kD_{2n}$-modules are projective and all the semibricks over $kD_{2n}$ are semisimple for any positive integer $n$.\\

\noindent (b) We assume $p:=\mathrm{char}\,k>0$ and let $SL_n(q)$ be the special linear group of degree $n$ over $\mathbb{F}_q$. Since $|Z(SL_n(q))|=\gcd(n,q-1)$, if $\gcd(n,q-1)$ is a power of $p$, then we have $\mathrm{s}\tau\textrm{-}\mathrm{tilt}\,kSL_n(q)\cong\mathrm{s}\tau\textrm{-}\mathrm{tilt}\,kPSL_n(q)$ by Theorem $\ref{mainthm}$. For the principal blocks, there exists an isomorphism $$\mathrm{s}\tau\textrm{-}\mathrm{tilt}\,B_0(kSL_n(q))\cong\mathrm{s}\tau\textrm{-}\mathrm{tilt}\,B_0(kPSL_n(q))$$
as partially ordered sets for any positive integer $n$ by Theorem \ref{genmainthm}.\\

\noindent (c) We assume $p:=\mathrm{char}\,k>0$ and let $G$ be a finite nilpotent group and $N$ a $p$-Sylow subgroup of $G$. Since $H(G)=G$, we have $\mathrm{s}\tau\textrm{-}\mathrm{tilt}\,kG\cong\mathrm{s}\tau\textrm{-}\mathrm{tilt}\,k(G/N)$ by Corollary $\ref{maincor}$. However, $k(G/N)$ is semisimple by Maschke's theorem. Therefore, all the support $\tau$-tilting $kG$-modules are projective and all the semibricks over $kG$ are semisimple.\\

\section*{Acknowledgement}

This work was supported by JST, the establishment of university fellowships towards the creation of science technology innovation, Grant Number JPMJFS2123. I would like to thank my supervisor Syu Kato for his thoughtful instructions. I would also thank Yuta Kozakai for his constructive suggestions.\\


\begin{thebibliography}{99}
\bibitem[A]{A} S. Asai, \textit{Semibricks}, Int. Math. Res. Not. IMRN, Volume 2020 (2020), Issue 16 , 4993--5054.
\bibitem[AI]{AI} T. Aihara and O. Iyama, \textit{Silting mutation in triangulated categories}, J. Lond. Math. Soc. (2), Volume 85 (2012), Issue 3, 633--668.
\bibitem[AIR]{AIR} T. Adachi, O. lyama and I. Reiten, \textit{$\tau$-tilting theory}, Compos. Math., Volume 150 (2014), Issue 3, 415--452. 
\bibitem[ARS]{ARS} M. Auslander, I. Reiten, and S. O. Smal\o, \textit{Representation theory of Artin algebras}, Cambridge Stud. Adv. Math., Volume 36, Cambridge University Press, Cambridge (1995). 
\bibitem[AS]{AS} M. Auslander and S. O. Smal\o, \textit{Almost split sequences in subcategories}, J. Algebra, Volume 69 (1981), Issue 2, 426--454.
\bibitem[B]{B} M. Broué, \textit{Isométries parfaites, types de blocs, catégories dérivées}, Astérisque, Issue 181--182 (1990), 61--92.
\bibitem[BY]{BY} T. Br\"{u}stle and D. Yang, \textit{Ordered exchange graphs}, Advances in Representation Theory of Algebras, EMS Ser. Congr. Rep., Eur. Math. Soc., Z\"{u}rich (2013), 135--193.
\bibitem[EJR]{EJR} F. Eisele, G. Janssens and T. Raedschelders, \textit{A reduction theorem for $\tau$-rigid modules}, Math. Z. 290 (2018), Issue 3--4, 1377--1413.
\bibitem[KK1]{KK1} R. Koshio and Y. Kozakai, \textit{On support $\tau$-tilting modules over blocks covering cyclic blocks}, J. Algebra, Volume 580 (2021), 84--103.
\bibitem[KK2]{KK2} R. Koshio and Y. Kozakai, \textit{Induced modules of support $\tau$-tilting modules and extending modules of semibricks over blocks of finite groups}, J. Algebra, Volume 628 (2023), 524--544.
\bibitem[KK3]{KK3} R.Koshio and Y. Kozakai, \textit{Normal subgroups and support $\tau$-tilting modules}, arXiv:2301.04963 (2023).
\bibitem[KY]{KY} S. Koenig and D. Yang, \textit{Silting objects, simple-minded collections, t-structures and co-t-structures for finite-dimensional algebras}, Doc. Math. 19 (2014), 403--438.
\bibitem[P]{P} R. S. Pierce, \textit{Associative algebras}, Springer-Verlag, New York-Berlin (1982).
\bibitem[R]{R} J. Rickard, \textit{Morita theory for derived categories}, J. London Math. Soc. (2), Volume 39 (1989), Issue 3, 436--456.
\bibitem[S]{S} P. Schneider, \textit{Modular representation theory of finite groups}, Springer, Dordrecht (2013).
\end{thebibliography}
\end{document}